\documentclass[a4paper,12pt]{article}

\newenvironment{proof}{\noindent {\bf Proof:}}{\hfill $\Box$}

\newtheorem{theorem}{Theorem}
\newtheorem{lemma}{Lemma}
\newtheorem{proposition}{Proposition}
\newtheorem{corollary}{Corollary}

\newtheorem{definition}{Definition}
\newtheorem{assumption}{Assumption}
\newtheorem{remark}{Remark}

\usepackage{amsmath}
\usepackage{amssymb}
\usepackage{amsfonts}
\usepackage{graphicx}
\usepackage{mathrsfs}

\newcommand{\R}{\mathbb{R}}
\newcommand{\N}{\mathbb{N}}

\newcommand{\SignedMeasures}{\mathcal{M}}
\newcommand{\PositiveMeasures}{\mathcal{M}^+}

\newcommand{\diff}{\mathrm{d}}

\title{\bf Modal occupation measures  and LMI  
relaxations  for   nonlinear 
switched systems  control  \footnote{This work with partly
supported by project 13-06894S of the Grant Agency of the Czech Republic and by the UK Engineering and Physical Sciences Research Council under Grant EP/G066477/1.}}

\begin{document}

\author{Mathieu Claeys$^1$,  Jamal Daafouz$^2$, Didier Henrion$^{3,4,5}$ }

\footnotetext[1]{Avenue Edmond Cordier 19, 1160 Auderghem, Belgium. {\tt mathieu.claeys@gmail.com}}
\footnotetext[2]{Universit\'e de Lorraine, CRAN, CNRS, IUF, 2 avenue de la for\^et de Haye, 54516 Vand\oe uvre cedex, France. {\tt jamal.daafouz@univ-lorraine.fr}}
\footnotetext[3]{CNRS; LAAS; 7 avenue du colonel Roche, F-31400 Toulouse; France. {\tt henrion@laas.fr}}
\footnotetext[4]{Universit\'e de Toulouse; LAAS; F-31400 Toulouse; France}
\footnotetext[5]{Faculty of Electrical Engineering, Czech Technical University in Prague,
Technick\'a 2, CZ-16626 Prague, Czech Republic}

\date{Draft of \today}

\maketitle

\begin{abstract}
This paper presents a linear programming approach for the optimal control of nonlinear switched systems
 where the control is the switching sequence.  This is done by introducing modal occupation measures, which allow to relax the problem as a primal linear programming (LP) problem. 
Its dual  linear program  of Hamilton-Jacobi-Bellman  inequalities  is also characterized.
The  LPs are then solved numerically with a
converging hierarchy of primal-dual moment-sum-of-squares (SOS) linear matrix inequalities (LMI). 
Because of the special structure of switched systems, we obtain a much more efficient method
than could be achieved by applying  standard moment/SOS LMI hierarchies for general optimal control problems. 
\end{abstract}

\section{Introduction}

A switched system is a particular class of a hybrid system that consists of a set of dynamical subsystems, one of which is active at any instant of time, and a policy for activating and deactivating the subsystems. These dynamical systems are encountered in a wide variety of application domains such as automotive industry, power systems, aircraft and traffic control, and more generally the area of embedded systems. Switched systems have been the concern of much research, and many results are available for stability analysis and control design. They make evident the importance of orchestrating the subsystems through an adequate switching strategy, in order to impose global stability  and/or performance. 
Interested readers may refer to the survey papers \cite{DeCarlo,Liberzonb,Shorten,Lin} and the books \cite{Liberzon,Sun} and the references therein. 

In this context, these systems are generally controlled by intrinsically discontinuous control signals, whose switching rule must be carefully designed to guarantee stability and performance properties. As far as optimality is concerned, several results are available in two main different contexts.

The first category of methods exploits necessary optimality conditions, in the form of Pontryagin's maximum principle (the so-called indirect approaches), or through a large nonlinear discretization of the problem (the so-called direct approaches). The first contributions can be found in \cite{Branicky,Hed99,Piccoli} where the problem has been formulated and partial solutions provided through generalized Hamilton-Jacobi-Bellman (HJB)  equations or inequalities and convex optimization. In \cite{Riedinger,Sussmann,Shaikh2007}, the maximum principle is generalized to the case of general hybrid systems with nonlinear dynamics. The case of switched systems is discussed in \cite{Bengea,Seatzu,Axelsson}. 
For general nonlinear problems and hence for switched systems, only \emph{local optimality} can be guaranteed , even when discretization can be properly controlled \cite{Xu2}. The subject is still largely open and we are far from a complete and numericaly tractable solution to the switched optimal control problem.

The second category collects extensions of the performance indices $H_2$ and $H_\infty$
originally developed for linear time invariant systems without switching, and use the flexibility
of Lyapunov's approach, see for instance \cite{Geromel,Deaecto} and references therein. Even for linear switched systems, the proposed results are based on nonconvex optimization problems (e.g., bilinear matrix inequality conditions) difficult to solve directly. Sufficient  convex  Linear Matrix Inequality (LMI) design
conditions may be obtained,  yielding computed solutions which are \emph{suboptimal},
 but at the price of introducing a conservatism (pessimism)   and a gap to optimality  which is
hard, if not impossible, to evaluate. 
Since the computation of this optimal strategy is a difficult task, a suboptimal solution is of interest  only when it is proved to be consistent, meaning that it imposes to the switched system a performance
not worse than the one produced by each isolated subsystem \cite{Gero}. 

Despite the interest of these existing approaches, the optimal control problem is not globally solved for switched systems, and new strategies are more than welcome, as computationally viable design techniques are missing. 
The present paper aims at complementing this rich literature on switched systems with convex programming techniques from the optimal control literature. Indeed, it is a well-known fact \cite{Vinter1978equivalence,Vinter1993ConvexDuality} that optimal control problems can be relaxed as linear programming (LP) problems over measure spaces. Despite some early numerical results \cite{Rubio1986book}, this approach was merely of theoretic interest, before the advent of tractable  primal-dual moment-sum-of-squares (SOS)
\emph{LMI hierarchies}, also now called Lasserre hierarchies, see \cite{lasserre} for an introduction to the subject. Most notably, \cite{sicon} explores in depth the application of such a numerical method for general polynomial optimal control problems, where the control is bounded. By modeling control and trajectory functions as \emph{occupation measures}, one obtains a convex LP formulation of the optimal control problem. This infinite-dimensional LP can be solved numerically and approximately with a hierarchy of convex finite-dimensional LMIs. 

In this paper, we consider the problem of designing optimal switching rules in the case of polynomial switched dynamical systems. Classically, we formulate the optimal control switching problem as an optimal control problem with artificial controls being functions of time, each valued in  the discrete set  $\{0,1\}$, and we relax it into a control problem with such controls now valued in  the interval  $[0,1]$. In contrast with existing approaches following this relaxation strategy, see, e.g., \cite{Bengea,Riedinger99},
relying on the \emph{local} Pontryagin maximum principle,
our aim is to apply the \emph{global} approach of \cite{sicon}. We show that by specializing the convex objects for switched systems, one obtains a numerical method for which the number of modes driving the system enters \emph{linearly} into the complexity of the problem, instead of exponentially. 
On the one hand, our approach follows the optimal control modeling framework.
On the other hand, it exploits the flexibility and computational efficiency of
the convex LMI framework. In contrast with most of the existing work on LMI methods, we have
a \emph{guarantee of global optimality}, in the sense that we obtain an asympotically converging
(i.e., with vanishing conservatism) hierarchy of lower bounds on the achievable performance.

\subsection*{Organization of the paper}

The paper is organized as follows. \S\ref{sec:measureLP} constructs the generalized moment problem associated to the optimal control of switched systems, and details some conic duality results. Theses results are used in \S\ref{sec:momentRelax} to obtain an efficient numerical method based on  the moment/SOS LMI hierarchy. \S\ref{sec:extract} then outlines a simple method to extract candidate optimal trajectories from moment data. \S\ref{sec:extensions} gathers several useful extensions to the problem, kept separate from the main body of the article so as to ease exposition and notations. Finally, \S\ref{sec:examples}
illustrates the method with several examples. 

\subsection*{Contributions}

The contributions of this paper are as follows. On a theoretical level, a novel way of relaxing optimal control problems of switched systems is proposed, following \cite{Henrion2013switch}. The equivalence of this convex program with existing results in the literature \cite{Vinter1978equivalence,Vinter1993ConvexDuality} is now rigorously proven. In addition, a full characterization of the dual conic problem as a system of  relaxed HJB  inequalities  is given. 

On a practical level, we present an easy to implement numerical method to solve the control problem, based on the 
LMI relaxations  for optimal control \cite{sicon}. The specialized measure LP formulation allows for substantial computational gains over generic formulations, which we now characterize. To our knowledge, this is one of the few results, along with \cite{impulse}  and \cite{ltvimpulse}, exploiting problem structure for more efficient 
LMI relaxations  of control problems.

\subsection*{Preliminaries}

Let $\SignedMeasures(X)$ denote the vector space of finite, signed, Borel measures supported on  an Euclidean subset
$X \subset \R^n$,  equipped with the weak-$*$ topology,  see, e.g.,  \cite{Royden} for background material. 
Let $\PositiveMeasures(X)$ denote the cone of non-negative measures in $\SignedMeasures(X)$. The support of $\mu \in \PositiveMeasures(X)$, the closed
set of all points $x$ such that $\mu(A)> 0$ for every open neighborhood
$A$ of $x$, is denoted by $\mathrm{spt} \, \mu$. For $\mu \in \SignedMeasures(X)$, $\mathrm{spt}\,\mu$ is the union of
the supports of the Jordan decomposition of $\mu$.
For a continuous function $f \in {\mathcal C}(X)$, we denote by $\int_{X} \! f(x) \, \mu( \diff x)$ the integral of $f$ w.r.t. the measure $\mu \in \SignedMeasures(X)$.
When no confusion may arise, we note $\langle f, \mu \rangle = \int f \mu$ for the integral to simplify exposition and to insist on the duality relationship between $ {\mathcal C}(X)$ and $\SignedMeasures(X)$. The Dirac measure supported at $x^*$, denoted by $\delta_{x^*}$, is the measure for which $\langle f, \delta_{x^*} \rangle = f(x^*)$ for all $f \in {\mathcal C}(X)$.  The indicator function of set $A$, denoted by $I_A(x)$,  is equal to one if $x \in A$, and zero otherwise. The space of probability measures on $X$, viz., the subset of $\PositiveMeasures(X)$ with mass $\langle 1,\mu \rangle = 1$, is denoted by $\mathcal{P}(X)$.

For multi-index $\alpha \in \N^n$, $\N$ being the set of natural numbers, and vector $x\in \R^n$, we use the notation $x^\alpha := \prod_{i=1}^n x_i^{\alpha_i}$. 
We denote by $\N^n_m$ the set of vectors $\alpha \in \N^n$ such that $\sum_{i=1}^n \alpha_i \leq m $. 
The moment of multi-index $\alpha \in \N^n$ of measure $\mu \in \PositiveMeasures(X)$ is then defined as the real $y_\alpha = \langle x^\alpha, \mu \rangle$.  

A multi-indexed sequence of reals $(y_\alpha)_{\alpha \in \N^n}$ is said to have a representing measure on $X$ if there exists $\mu \in \PositiveMeasures(X)$ such that $y_\alpha = \langle x^\alpha , \mu \rangle$ for all $\alpha \in \N^n$.

Let $\R[x]$ denote the ring of polynomials in the variables $x \in \R^n$, and let $\deg p$ denote the (total) degree of polynomial $p$.

A subset $X$ of $\R^n$ is basic semi-algebraic if it is defined as the intersection of finitely many polynomial inequalities, 
namely $X = \lbrace x \in \R^n : \; g_i(x) \geq 0, \, i = 1 \ldots n_{X}\rbrace$ with $g_i(x) \in \R[x], \, i = 1 \ldots n_{X}$.

\section{Weak linear program}
\label{sec:measureLP}
\subsection{Problem statement}
Consider the optimal control problem
\begin{equation}\label{ocp}
\begin{array}{ll}
\inf\limits_{\sigma} & \displaystyle\int_0^T l_{\sigma(t)}(t,x(t)) \, \diff t \\
\mathrm{s.t.} & \dot{x}(t) = f_{\sigma(t)}(t,x(t)),\\
& x(0) = x_0,  \, x(T) = x_T,  \\
& (t,x(t)) \in [0,T] \times X,  \\
&  \sigma(t) \in \{1,2,\ldots,m\},\\
\end{array}
\end{equation}
where the infimum is over the sequence $\sigma$. The state $x$ is a vector of $n$ elements and the control consists
of choosing, for almost all times, which of the $m$ possible dynamics drives the system.
Throughout this section, the following assumptions hold:
\begin{assumption}
\label{th:standingAssumptions}
Terminal time $T$, initial state $x_0$ and terminal state $x_T$ are given.
State constraint set $X$ is given and compact. Let $K:=[0,T]\times X$. 
Functions $f_j:K \mapsto \R^n$ (dynamics), and $l_j:K \mapsto \R$, $j=1,\dots,m$
 (Lagrangians) are given, continuous in all their variables, and Lipschitz continuous with respect to the state variable.
\end{assumption}
The integrand in the cost is called a Lagrangian, following the standard optimal control terminology, see e.g., \cite{Clarke2013Functional}.
Note that the above assumptions can be significantly weakened. For instance, mere measurability of
the dynamics w.r.t. time, lower semi-continuity of the Lagrangian or unbounded state spaces when the Lagrangian is coercive are often considered in the calculus of variation literature. However, as \S\ref{sec:momentRelax} explores in depth, the question of the numerical representation of problem data is essential for rigorous \emph{global} optimization. Up to our knowledge, only polynomials can offer such guarantees, and Assumption \ref{th:standingAssumptions} is general enough to cover such an essential case.

Obviously, optimal control problem \eqref{ocp} can be equivalently written as
\begin{equation}\label{discrete}
\begin{array}{ll}
\inf\limits_{u} & \displaystyle\int_0^T \sum_{j=1}^m l_j(t,x(t)) \, u_j(t) \, \diff t \\
\mathrm{s.t.} & \displaystyle \dot{x} = \sum_{j=1}^m f_j(t,x(t)) \, u_j(t)\\
& x(0) = x_0, \, x(T) = x_T,\\
& (t,x(t)) \in K, \, u(t) \in U
\end{array}
\end{equation}
where the infimum is with respect to a time-varying vector $u(t)$
which belongs for all $t \in [0,T]$ to the (nonconvex) discrete set
\begin{equation}\label{eq:U}
U := \{(1,0,\ldots,0), (0,1,\ldots,0), \ldots, (0,0,\ldots,1)\} \subset {\mathbb R}^m.
\end{equation}
The next sections explore the various ways this problem can be lifted/relaxed as a linear program over so-called \emph{occupation measures}. 

\begin{remark}
Section \ref{sec:extensions} gathers several extensions to problem \eqref{ocp} for which the present approach is applicable mutatis mutandis. This section is relegated at the end of the paper for ease of exposition.
\end{remark}

\subsection{Occupation measures}
\label{sec:occupMeas}
Simple examples (see for instance at the end of the paper) show that when measurable functions of time are considered for controls such as in \eqref{discrete}, optimal solutions might not exist. This section summarizes several well-known
concepts to regain compactness in the control. These tools will then be specialized in the next section for switched systems.

\begin{definition}[Young measure]
A Young measure is a  time-parametrized  family of probability measures $\omega(\diff u | t) \in  \mathcal{P}(U)$ defined almost everywhere, such that for all continuous functions\footnote{Considering measures as a subset of distributions, those may be referred to as test functions} $v \in {\mathcal C}(U)$, the function $ t \rightarrow \langle v, \omega \rangle$ is Lebesgue measurable. 
\end{definition}
Young measures are also called generalized controls  or relaxed controls,  see for instance \cite[Part III]{Fattorini}.
 Note that classical controls $u(t)$, functions of time with values in $U$, are a particular case of Young measures
for the choice $\omega(\diff u | t) = \delta_{u(t)}(\diff u)$. 
Following the terminology of  \cite{Vinter1993ConvexDuality}, it is natural to relax \eqref{discrete} into the following ``strong form''\footnote{``Strong'' is opposed here to the ``weak'' problem to be defined shortly after.}:
\begin{equation}\label{convex}
\begin{array}{rcll}
p^*_S & = & \inf\limits_{\omega} & \displaystyle\int_0^T \sum_{j=1}^m l_j(t,x(t))u_j(t) \, \diff t \\
&& \mathrm{s.t.} & \dot{x} =  \displaystyle\left\langle  \sum_{j=1}^m f_j(t,x(t)) u_j, \, \omega(\diff u | t) \right\rangle \\
&&& x(0) = x_0, \, x(T) = x_T, \, (t,x(t)) \in K
\end{array}
\end{equation}
 where the infimum is w.r.t. a Young measure $\omega(\diff  u | t) \in \mathcal{P}(U)$. 
Note that there might be a strict relaxation gap, viz.,  $p^*_S$ is strictly less than the infimum in problem (\ref{ocp}), for ill-posed instances. It is not the purpose of this article to explore such a non-generic case, and we refer to the discussion in, e.g., \cite[App.~C]{roa} and the references therein on such occurrences. There are sufficient conditions guaranteeing the absence of such a gap, referred to as inward-pointing conditions,
see for instance \cite{Frankowska}.

An equivalent way of relaxing \eqref{discrete} is to relax the controls to belong, for all $t \in [0,T]$, to the (convex) simplex
\begin{equation}\label{eq:simplex}
\mathrm{conv}\:U = \{u \in {\mathbb R}^m:\: \sum_{j=1}^m u_j = 1, \; u_j \geq 0, \; j=1,\ldots,m\}.
\end{equation}
In other words, problem (\ref{convex}) on Young measures is equivalent to
replacing $U$ with $\mathrm{conv}\:U$ in problem (\ref{discrete}). 
In \cite{Bengea}, problem (\ref{convex}) is called the embedding of problem (\ref{ocp}),
and it is proved that the set of trajectories
of problem (\ref{ocp}) is dense (w.r.t. the uniform norm in the space of continuous
functions of time) in the set of trajectories
of embedded problem (\ref{convex}). Note however that these authors consider
the more general problem of switching design in the presence
of additional bounded controls in each individual dynamics.
To cope with chattering effects due to the simultaneous presence
of controls and (initial and terminal) state constraints,
they have to introduce a further relaxation of the embedded
control problem. In this paper, we do not have controls in
the dynamics, and the only design parameter is the switching
sequence.

Notice that once a Young measure $\omega(\diff u | t)$ is given in problem (\ref{convex}), the
corresponding state trajectory is uniquely determined by
\[
x(t) = x_0 + \int_0^t \left\langle  \sum_{j=1}^m f_j(s,x(s)) u_j, \, \omega(\diff u | s) \right\rangle \diff s
\]
and hence the following definition makes sense.

\begin{definition}[Relaxed arc]
A pair $(x,\omega)$ which is feasible, or admissible, for problem \eqref{convex} is called a relaxed arc.
\end{definition}

For optimization, strong problem \eqref{convex} presents an essential first step, as sequential-compactness is regained, guaranteeing the existence of optimal solutions under the sole existence of a relaxed arc \cite[Th.~1.2]{Vinter1993ConvexDuality}. For \emph{global} optimization though, the nonlinear dependence of \eqref{convex} on trajectories $x(t)$ is problematic, and an additional generalization of Young measures must be introduced, capturing both controls \emph{and} trajectories:

\begin{definition}[Occupation measure]
\label{th:occupMeas}
Given  a relaxed arc $(x,\omega)$,  its corresponding occupation measure
is defined by
\[
\mu(A\times B \times C ) := \int_{A} \!\!\! I_B(x(t) ) \omega(C | t) \, \diff t
\]
for all Borel subsets $A,B,C$ of $[0,T]$, $X$ and $U$, resp.
\end{definition}
That is, occupation measures assign to each set $A \times B \times C$ the total time ``spent'' on it by 
a relaxed arc $(x,\omega)$.  An alternative, equivalent definition could be
\[
\diff\mu(t,x,u) = \delta_{x(t)}(\diff x)\omega(\diff u | t)\diff t.
\]

By construction, an occupation measure satisfies
\begin{equation}\label{eq:weakdef}
\langle v, \mu \rangle = \int_0^T \left(\int_U \!\! v(t,x(t),u) \omega(\diff u | t) \right)\diff t
\end{equation}

for all continuous test functions $v \in {\mathcal C}([0,T] \times X \times U)$.
Indeed, by the Riesz representation theorem (see, e.g., \cite[\S{}21.5]{Royden}), an occupation measure could alternatively be defined as the unique measure satisfying \eqref{eq:weakdef} for all continuous test functions (recall the hypothesis of compact supports).

Evaluating now a continuously differentiable test function along a trajectory and making use of \eqref{eq:weakdef} reveals the following  lemma:
\begin{lemma}
\label{th:weakOdeOccMeas}
 Given a relaxed arc $(x,\omega)$,
its corresponding occupation measure satisfies the following \emph{linear} constraints for any test function $v \in {\mathcal C}^1(\R \times \R^n)$:
\begin{equation}\label{eq:weakDynamics}
\left\langle \frac{\partial v}{\partial t} + \frac{\partial v}{\partial x} \cdot \sum_{j=1}^m \,f_j \, u_j , \mu \right\rangle = v(T,x_T) - v(0,x_0).
\end{equation}
\end{lemma}
Note the use of a scalar product, denoted by a dot, between the gradient vector $\frac{\partial v}{\partial x}$ 
and the controlled dynamics. In \cite{Rubio1986book}, it is proposed then to relax strong problem \eqref{convex} by optimizing over all Borel measures satisfying \eqref{eq:weakDynamics} instead of simply occupation measures, yielding the following ``weak'' problem:
\begin{equation}\label{eq:weak}
\begin{array}{rcll}
p^*_W & = & \inf\limits_{\mu} & \displaystyle\left\langle \sum_{j=1}^m l_j(t,x(t))u_j(t) , \mu \right\rangle  \\
&& \mathrm{s.t.} & \forall v \in {\mathcal C}^1(K): \; v(T,x_T) - v(0,x_0) = \\
&&& \qquad \displaystyle\left\langle \frac{\partial v}{\partial t} + \frac{\partial v}{\partial x} 
\cdot \sum_{j=1}^m \,f_j \, u_j , \mu \right\rangle  \\
&&& \mu \in \PositiveMeasures(K \times U).
\end{array}
\end{equation}
Weak problem \eqref{eq:weak} is rigorously justified by the following essential result:
\begin{theorem}
\label{th:noGapVinter}
There is no relaxation gap between optimal control problem \eqref{convex} and  
measure LP \eqref{eq:weak}, viz.,
\[
p^*_S = p^*_W.
\]
\end{theorem}
\begin{proof}
This result is not new, so its proof is only sketched.
Obviously, direction $p^*_S \geq p^*_W$ is a simple consequence of Lemma~\ref{th:weakOdeOccMeas} and the continuity of the cost. The other direction, more involved, is the core of \cite{Vinter1978equivalence} and especially \cite[Theorem 2.3]{Vinter1993ConvexDuality}.
Synthetically, the authors use specialized Hahn-Banach type separation arguments to show that the admissible set for weak problem \eqref{eq:weak} is the closure of the convex hull of the set of occupation measures of the admissible solutions of \eqref{convex}. That is, one could interpret those results as a Krein-Milman theorem for optimal control problems, see e.g.
\cite[Chapter 14]{Royden} for background on the Krein-Milman theorem.
\end{proof}

Note that the use of relaxations and LP formulations of optimal control
problems (on ordinary differential equations and partial differential equations)
is classical, and can be traced back to the work by L. C. Young, Filippov, as well as Warga and Gamkrelidze, amongst many others. For more details and a historical
survey, see, e.g., \cite[Part III]{Fattorini}.

\subsection{Modal occupation measures and primal LP}
Occupation measures defined in \S\ref{sec:occupMeas} are supported on a subset of Euclidean space of dimension $n+m+1$. As such, they prove to be challenging for practical numerical resolution when the size of the problem increases, a phenomenon known colloquially in the (dual) dynamic programming literature as the ``curse of dimensionality''.

For instance, the numerical approach of \cite{Rubio1986book} proposes to consider a finite subset of the countably many linear constraints on $\mu$ given by \eqref{eq:weakDynamics}. Then, as a consequence of Tchakaloff's theorem \cite[Th.~B.12]{lasserre}, there exists an atomic measure with support included in the support of $\mu$, and satisfying the truncated linear constraints exactly. By gridding the support of $\mu$ with Dirac measures of unknown masses, one can then approach weakly-$*$ this atomic measure as the spatial resolution of the grid gets finer. This results in a finite-dimensional LP to solve, with a number $ \mathcal{O} \left( 1/\Delta^{n+m+1} \right)$ of unknowns, with $\Delta$ the grid resolution. Therefore, for a fixed resolution, the size of the LP grows exponentially in the size of the state and control spaces.

Other examples suffering a similar fate are LMI relaxations as presented later on in \S\ref{sec:momentRelax}, which the authors prefer over LP approximations for their greater rigor -- most noticeably the guarantee of obtaining lower bounds
 and their better convergence properties, see \cite[Section 5.4.2]{lasserre}. 
Indeed, as shown in \S\ref{sec:complexity}, for a given relaxation order, the size of the semi-definite blocks
in the LMIs also grows exponentially with the dimension of the underlying space.

For finite-dimensional optimization problems, structural properties of the problem can be exploited to define equivalent LPs using more measures, but of smaller dimensions, see \cite{Lasserre2006sparsity}. This allows for much faster numerical resolution.
Motivated by these savings, we now introduce specialized occupation measures to satisfy much the same goal for the optimal control of switched systems: 
\begin{definition}[Modal occupation measures]
\label{th:modalOccupMeas}
The $j$-th modal occupation measure $\mu_j(\diff t, \diff x) \in \PositiveMeasures(K)$ associated with occupation measure $\mu$  as in Def.~\ref{th:occupMeas}, is defined as
\[
    \mu_j(A \times B) :=
    \int_A \int_B \int_U  u_j \, \diff\mu(t,x,u),
\]
for all Borel subsets $A, B$ of $[0,T]$, $X$, resp.
\end{definition}
In words, modal occupation measures are the {\it projections} on $K=[0,T]\times X$ of the first moments of the occupation measure defined in \S\ref{sec:occupMeas} w.r.t. the controls. For each Borel subset of $K$, each modal occupation measure therefore specifies the time spent by the pair $(t,x(t))$ on each mode. As a consequence, Def.~\ref{th:modalOccupMeas} implies the following specialization of Lem.~\ref{th:weakOdeOccMeas} for switched systems:
\begin{corollary}
\label{th:weakModalOdeOccMeas}
 Given a relaxed arc $(x,\omega)$, 
its corresponding modal occupation measures satisfy the following \emph{linear} constraints:
\begin{equation}\label{eq:weakDynamicsModal}
\sum_{j=1}^m \left\langle \frac{\partial v}{\partial t} + \frac{\partial v}{\partial x}  \cdot f_j  , \mu_j \right\rangle = v(T,x_T)- v(0,x_0)
\end{equation}
for all test functions $v \in {\mathcal C}^1(K)$.
\end{corollary}
\begin{proof}
Notice that by definition of $U$, the occupation measure satisfies
\[
\left\langle \frac{\partial v}{\partial t}, \mu \right\rangle = \left\langle \frac{\partial v}{\partial t} \cdot \sum_{j=1}^m u_j, \mu \right\rangle.
\]
The corollary is then the immediate application of Def.~\ref{th:modalOccupMeas} and Fubini's theorem. 
\end{proof}

Corollary \ref{th:weakModalOdeOccMeas} naturally points to the following alternative weak problem for switched systems:
\begin{equation}\label{eq:weakSwitchedAlternative}
\begin{array}{rccl}
p^* & = & \inf\limits_{(\mu_1,\ldots,\mu_m)} & \displaystyle  \sum_{j=1}^m  \left\langle l_j  , \mu_j \right\rangle  \\
& & \mathrm{s.t.} &  \forall v \in {\mathcal C}^1(K): \; v(T,x_T) - v(0,x_0) = \\
& & & \quad \displaystyle \sum_{j=1}^m \left\langle \frac{\partial v}{\partial t} + \frac{\partial v}{\partial x}   f_j , \mu_j \right\rangle  \\
& & & \mu_j \in \PositiveMeasures(K), \quad j=1,\ldots , m \\               
\end{array}   
\end{equation}

where the minimization is w.r.t. the vector of modal occupation measures $(\mu_1,\ldots,\mu_m)$.

This is rigorously justified by the following:
\begin{theorem}
\label{th:noGapModal}
There is no relaxation gap between measure LP \eqref{eq:weak}
and measure LP \eqref{eq:weakSwitchedAlternative}, viz.,
\begin{equation*}
p^* = p_{W}^*
\end{equation*}
\end{theorem}

\begin{proof}
We first show $p^* \leq p_{W}^*$. Consider an admissible measure $\mu$ for \eqref{eq:weakSwitchedAlternative}. Then, by Cor.~\ref{th:weakModalOdeOccMeas}, there exists an admissible vector $(\mu_1, \ldots \mu_m)$ satisfying \eqref{eq:weakDynamicsModal}.
As the cost can be treated in an analogous fashion, this shows that any admissible solution for 
\eqref{eq:weakSwitchedAlternative} is also admissible for \eqref{eq:weak} with the same cost, hence $p^* \leq p_{W}^*$.

For the reverse inequality, consider now a vector $(\mu_1,\ldots,\mu_m)$ admissible for \eqref{eq:weak}, and define
\begin{equation}\label{eq:defSumMeas}
\bar{\mu} := \sum_{j=1}^m \mu_j.
\end{equation}

By construction, each $\mu_j$ is absolutely continuous with respect to $\bar{\mu}$, such that by the Radon-Nikod\'ym Theorem \cite[\S 18.4]{Royden}, there exists a non-negative, measurable function  $\bar{u}_j(t,x)$  such that $\mu_j = \bar{u}_j \, \bar{\mu}$. Injecting this back in \eqref{eq:defSumMeas} shows that, in addition, $\sum_{j=1}^m \bar{u}_j = 1$, $\bar{\mu}$-almost everywhere.
Define now $\mu \in \PositiveMeasures(K \times U)$, with $U$ as in \eqref{eq:U}, such that
\[
    \diff\mu(t,x,u) : = \left( \sum_{j=1}^m \bar{u}_j(t,x) \delta_1(\diff u_j | t,x)  \right) \diff\bar{\mu}(t,x),
\]
such that by definition
\[
    \int_U \! u_j \, \diff\mu(t,x,u)  =  \bar{u}_j(t,x) \, \diff\bar{\mu}(t,x),
\]
where the last two equations are understood in a weak sense.
Working out the weak dynamics of \eqref{eq:weakSwitchedAlternative} leads then to:
\begin{align}
\sum\limits_{j=1}^m  \left\langle \frac{\partial v}{\partial t} + \frac{\partial v}{\partial x} \cdot f_j , \mu_j \right \rangle & =
\sum\limits_{j=1}^m \left\langle \left( \frac{\partial v}{\partial t} + \frac{\partial v}{\partial x} \cdot f_j \right) \bar{u}_j , \bar{\mu}  \right\rangle \notag \\
& = \left\langle \frac{\partial v}{\partial t} + \sum\limits_{j=1}^m \frac{\partial v}{\partial x} \cdot f_j  \, \bar{u}_j , \bar{\mu} \right\rangle \\
& = \left\langle \frac{\partial v}{\partial t} + \sum\limits_{j=1}^m \frac{\partial v}{\partial x} \cdot f_j \, u_j , \mu \right\rangle. \notag
\end{align}
As the cost can be treated in a similar manner, this concludes $p^* \geq p_{W}^*$, hence the proof.
\end{proof}

Theorem \eqref{th:noGapModal} fulfills the objective laid at the beginning of this section. Indeed structured LP \eqref{eq:weakSwitchedAlternative} involves $m$ modal occupation measures supported on spaces of dimension $n+1$, instead of a single occupation measure on a space of dimension $n+m+1$ as in unstructured LP \eqref{eq:weak}.

\subsection{HJB inequalities and dual LP}
\label{sec:duality}

Before investigating the practical implications on LMI relaxations of structured measure LP \eqref{eq:weakSwitchedAlternative}, we explore its conic dual. This is an interesting result in its own right for so-called ``verification theorems'' which supply necessary and sufficient conditions in the form of more traditional HJB  inequalities. 
In addition, practical numerical resolution of the LMI relaxations by primal/dual interior-point methods \cite{mosek} implies that a strengthening of this dual will be solved as well, as shown later on in \S\ref{sec:SOS}. 

In this section, it is first shown that the solution of \eqref{eq:weakSwitchedAlternative} is attained whenever an admissible solution exists. Then, the dual problem of \eqref{eq:weakSwitchedAlternative}, in the sense of conic duality, is presented. This leads directly to a system of subsolutions of HJB equations. Although, the cost of this problem might not be attained, it is however shown that there is no duality gap between the conic programs.

First of all, we establish that whenever the optimal cost of \eqref{eq:weakSwitchedAlternative} is finite, there exists a vector of measures attaining the value of the problem:
\begin{lemma}
\label{th:solAttained}
 If $p^*$ is finite in problem (\ref{eq:weakSwitchedAlternative}),
there exists an admissible $(\mu_1^*,\ldots,\mu_m^*)$ attaining the infimum, viz., such that
\[
  \sum_{j=1}^m \left\langle l_j , \mu_j^* \right\rangle  =  p^*.
\]
\end{lemma}

\begin{proof}
We first show that the mass of each measure is bounded. This is true for each $\mu_j$, since test function $v(t,x)=t$ imposes $\sum_{j=1}^m \langle 1 , \mu_j \rangle = T$. Then, following Alaoglu's theorem \cite[\S 15.1]{Royden}, the unit ball in the vector space of compactly supported measures is compact in the weak-$*$ topology. Therefore, any sequence of admissible solutions for \eqref{eq:weakSwitchedAlternative} possesses a converging subsequence. Since this must be true for any sequence, this is true for any minimizing sequence, which concludes the proof.
\end{proof}

Now, remark that \eqref{eq:weakSwitchedAlternative} can be seen as an instance of a conic program, called hereafter the primal,  in standard form (see for instance \cite{Barvinok2002convexity}):
\begin{equation}\label{eq:primal}
\begin{array}{rcll}
p^* & = & \displaystyle \inf\limits_{z_p} & \displaystyle \langle  z_p, c \rangle_p  \\
&& \mathrm{s.t. }  & \mathcal{A} \, z_p = b, \\
&&& x_p \in  E^+_p
\end{array}
\end{equation}
where
\begin{itemize}
\item  decision variable
$z_p := (\mu_1,\ldots,\mu_m)$
belongs to vector space
$E_p := \SignedMeasures(K)^m $
equipped with the weak-$*$ topology;

\item the cost
$c := (l_1,\ldots,l_m)$ belongs to vector space $F_p := {\mathcal C}(K)^m $ equipped
with the strong topology;

\item $\langle \cdot , \cdot \rangle_p: \; E_p \times F_p \rightarrow \R$ is the duality bracket given by the integration of continuous functions with respect to Borel measures,  since $E_p = F'_p$, with the prime denoting the topological dual; 
\item  the cone $E^+_p$ is the non-negative orthant of $E_p$; 

\item the linear operator $\mathcal{A}:E_p \rightarrow {\mathcal C}^1(K)'$ is given by
\[
z_p \mapsto \mathcal{A} z_p =  \sum_{j=1}^m \mathcal{L}_j \mu_j
\]
where $\mathcal{L}_j$ is the adjoint operator of $\mathcal{L}'_j: {\mathcal C}^1(K) \rightarrow {\mathcal C}(K)$ defined by
\[
v \mapsto \mathcal{L}'_j v := \frac{\partial v}{\partial t} +  \frac{\partial v}{\partial x} \cdot f_j;
\]

\item right hand side
$b := \delta_{(T,x_T)}(\diff t, \diff x) - \delta_{(0,x_0)}(\diff t, \diff x)$
belongs to vector space
$E_d :=  {\mathcal C}^1(K)'$.

\end{itemize}

\begin{lemma}
The conic dual of LP problem \eqref{eq:weakSwitchedAlternative} is given by
\begin{equation}\label{eq:dualExplicit}
\begin{array}{rcll}
d^* & = & \sup\limits_{v} & v(T,x_T) - v(0,x_0)  \\
&& \mathrm{s.t.}  & l_j - \mathcal{L}'_j v \geq 0, \:\:\forall \: (t,x) \in  K, \: j = 1 \ldots m, \\
&&& v \in {\mathcal C}^1(K).
\end{array}
\end{equation}
\end{lemma}

\begin{proof}
Following standard results of conic duality (see \cite{Anderson} or \cite{Barvinok2002convexity}), the conic dual of \eqref{eq:primal} is given by
\begin{equation}\label{dual}
\begin{array}{rcll}
d^* & = & \sup\limits_{z_d} & \langle b, z_d \rangle_d  \\
&& \mathrm{s.t.}  & c-\mathcal{A}' z_d \in {\mathcal C}^+(K)^m, \\
&&& z_d \in F_d
\end{array}
\end{equation}
where
\begin{itemize}
\item  decision variable $z_d :=  v$ belongs to vector space $F_d := {\mathcal C}^1(K)$;

\item  the (pre-)dual cone is ${\mathcal C}^+(K)^m$, i.e., $E^+_p$ is dual to ${\mathcal C}^+(K)^m$; 

\item $\langle \cdot , \cdot \rangle_d: \; E_d \times F_d \rightarrow \R$ is the appropriate duality pairing between $E_d$ and $F_d$.
\end{itemize}
The lemma just details this dual problem.
\end{proof}

Once the duality relationship established, the question arises of whether a duality gap occur between linear problems \eqref{eq:weakSwitchedAlternative} and \eqref{eq:dualExplicit}. The following theorem discards such a possibility:

\begin{theorem}
\label{th:noGap}
There is no duality gap between primal LP \eqref{eq:weakSwitchedAlternative} and dual LP \eqref{eq:dualExplicit}: if there is an admissible vector for \eqref{eq:weakSwitchedAlternative}, then
\[
p^* = d^*.
\]
\end{theorem}

\begin{proof}
One can show, following \cite[Th. 3.10]{Anderson} (see also the exposition in \cite[\S 4]{Barvinok2002convexity}), that the closure of the cone
\[
R := \left\lbrace (\langle z_p, c \rangle_p , \mathcal{A} \, z_p) : \; z_p \in E^+_p \right\rbrace
\]
belongs to $\R \times E_d$. Note that the weak-$*$ topology is implicit in the definition of $E_d$, and the closure of $R$ should be understood in such a weak sense.

To prove closure, one may show that, for any sequence of admissible solutions $(z_p^{(n)})_{n \in \N}$, all accumulation points of $(\langle z_p^{(n)}, c \rangle_p , \mathcal{A} \, z_p^{(n)})_{n \in \N}$ belong to $R$. Note that Lemma \ref{th:solAttained} establishes that any sequence $(z_p^{(n)})_{n \in \N}$ has a converging subsequence. Therefore, all that is left to show is the weak-$*$ continuity of $\mathcal{A}$, hence of each $\mathcal{L}_j$. Following \cite{sicon}, this can be shown by noticing that
each $\mathcal{L}'_j$ is continuous for the strong topology of ${\mathcal C}^1(K)$, hence for its associated weak topologies.
Operators $\mathcal{L}'_j$, hence $\mathcal{A}$, are therefore weakly-$*$ continuous, and each sequence $(\langle x_p^{(n)}, c \rangle_p , \mathcal{A} x_p^{(n)})_{n \in \N}$ converges in $R$, which concludes the proof.
\end{proof}

Note that what is not asserted in Th.~\ref{th:noGap} is the existence of a continuous function for which the optimal cost is attained in dual problem (\ref{eq:dualExplicit}). Indeed, it is a well-known fact that value functions of optimal control problems, to which $v$ is closely related, may not be continuous, let alone continuously differentiable. However, there does exist an admissible vector of measures for which the optimal cost of primal \eqref{eq:weakSwitchedAlternative} is attained (whenever there exists an admissible solution), following Lemma \ref{th:solAttained}. This gives practical motivations for working in a purely primal approach on measures, as the dual problem will be solved anyway as a side product, see in the later sections of this article.

We close this section by detailing how dual problem \eqref{eq:dualExplicit} can nonetheless be used analytically as a ``verification theorem'', following the discussion of \cite{Vinter1978equivalence} for the unstructured case, that is following the dual of \eqref{eq:weakSwitchedAlternative} instead.

Denote by $\tau_j$ the time subset for which the $j$-th mode is ``active'':
\begin{equation}\label{eq:tauj}
  \tau_j := \mathrm{spt} \pi_t \mu_j,
\end{equation}
where $\pi_t \mu_j$ is the projection, or marginal of $\mu_j$ on the time axis, viz., the unique measure such that
\[
 \langle v(t), \pi_t \mu_j \rangle = \int_{K} \!\! v(t) \, \diff\mu_j(t,x)
\]
for all continuous test functions $v$ and $j=1,\ldots,m$. If an admissible control $u(t)$ is piecewise-continuous, then obviously $\tau_j := \int_0^T \!\! I_{ \lbrace e_j \rbrace}(u(t)) \, \diff t$, where $e_j$ is the $j$-th unit vector of $\R^m$. 
Note however that the intersection of several of the $\tau_j$ might be of positive Lebesgue measure. In this case, relaxed controls can only be weakly approximated by fast oscillating sequences of classical controls. In the language of differential inclusions, this consists of realizing a control in the convexified vector field but not belonging to the original vector field.

\begin{corollary}
\label{th:verification}

Let $(x,\omega)$ be a relaxed arc,
and let the intervals $\tau_j$ be computed as in \eqref{eq:tauj} from their occupation measures. Then there exists
a sequence $(v_n)_{n\in \N} \in {\mathcal C}^1(K)$ admissible for \eqref{eq:dualExplicit}, such that
\[
\lim_{n \rightarrow \infty} \int_{\tau_j} \!\!\! \left( l_j(t,x(t))- \mathcal{L}'_j v_n(t,x(t)) \right) \diff t = 0, \: j = 1 \ldots m, 
\]
if and only if $(x,\omega)$ is a globally optimal solution of optimal control problem \eqref{convex}.
\end{corollary}
\begin{proof}
This corollary explicits weak duality (via the complementarity condition $\langle z_p^*, c-\mathcal{A}' z_d^* \rangle_p = 0$) if $z_p^*$ and $z_d^*$ are optimal for their respective problems, and exploits Lem.~\ref{th:solAttained} guaranteeing the existence of an optimal $z_p^*$.
\end{proof}

The advantage of such specialized optimality certificates for switched systems over the general case is that the test needs only to be checked on a mode-by-mode basis. This is particularly true if mode-dependent constraints are introduced, see \S\ref{sec:extraconstraints}.

\section{LMI relaxations}
\label{sec:momentRelax}

This section outlines the numerical method to solve \eqref{eq:weakSwitchedAlternative} in practice, by means of primal-dual
moment-SOS LMI relaxations.
This is done by restricting the problem data to be polynomial, so that occupation measures can be equivalently manipulated by their moments for the problem at hand, see \S\ref{sec:momentLP}. Then, this new infinite-dimensional problem is truncated so as to obtain in \S\ref{solvelp} a convex, finite-dimensional relaxation, with dual described in \S\ref{sec:SOS}.
Finally, \S\ref{sec:complexity} shows how the structural properties of \eqref{eq:weakSwitchedAlternative} improves the computational load of each of these relaxations in comparison to unstructured relaxations, as could be derived from \eqref{eq:weak}.

In the remainder of the paper, the standing Assumption~\ref{th:standingAssumptions} is therefore strengthened as follows:
\begin{assumption}
\label{th:polynomialDataAssumption}
All functions are polynomial in their arguments:
\[
   l_j, f_j \in \R[t,x], \quad j=1,\ldots,m.
\]
In addition,
set $K=[0,T]\times X$ is basic semi-algebraic,
viz., it is defined as the conjunction
of finitely many polynomial inequalities
\[
 K = \left\lbrace (t,x) \in \R \times \R^n : \, g_i(t,x) \geq 0, i = 1,\ldots,n_K   \right\rbrace
\]
with $g_i \in \R[t,x]$, $ i = 1,\ldots,n_K$. 
Let $g_0(t,x)$ denote the polynomial identically equal to one. 
In addition, it is assumed that one of the $g_i$ enforces a ball constraint on all variables, which is possible w.l.g. since
$K$ is assumed compact (see \cite[Th.~2.15]{lasserre} for slightly weaker conditions).
\end{assumption}

\subsection{Moment LP}
\label{sec:momentLP}

Given a sequence $y=(y_\alpha)_{\alpha \in \N^n}$, let $L_y:\R[x]\to\R$ be the Riesz linear functional
\[
 f(x)= \sum_\alpha f_\alpha x^\alpha \in\R[x]
  \quad \mapsto \quad
  L_y(f)\,=\,\sum_\alpha f_\alpha y_\alpha. 
\]
Define the {\it localizing matrix} of order $d$ associated with sequence $y$ and
polynomial $g(x) = \sum_\gamma g_\gamma x^\gamma \in \R[x]$ as the real symmetric matrix $M_d(g\,y)$ whose entry $(\alpha,\beta)$ reads
\begin{align}
  \label{eq:locMomMat}
 [M_d(g\,y)]_{\alpha, \beta} & = L_y \left( g(x) \, x^{\alpha+\beta} \right) \\
                    & = \sum_\gamma g_\gamma \, y_{\alpha+\beta+\gamma},
  \quad \forall \alpha,\beta \in \N^n_d.
\end{align}
In the particular case that $g(x)=1$, the localizing matrix is called the {\it moment matrix}.

The construction of the LMI associated with problem \eqref{eq:weakSwitchedAlternative} can now be stated. Its decision variable is the aggregate sequence of moments
$y=(y_1,y_2,\ldots,y_m)$ where each $y_j=(y_{j,\beta})_{\beta \in \N^{n+1}}$ is the moment sequence of measure $\diff\mu_j(t,x)$ for $j=1,2,\ldots,m$, i.e.,
\begin{equation}\label{mom}
y_{j,\beta} = \left\langle v_\beta(t,x), \mu_j \right\rangle
\end{equation}
for the  particular choice of test function 
\begin{equation}\label{eq:betaTestFun}
v_\beta(t,x) := t^{\beta_0} x_1^{\beta_1} x_2^{\beta_2} \cdots x_n^{\beta_n}, \:\:\beta \in \N^{n+1}.
\end{equation}

As each term of the cost of \eqref{eq:weakSwitchedAlternative} is polynomial by assumption, they can be rewritten as
\[
\sum_{j=1}^m \sum_{\beta \in \N^{n+1}} c_{j,\beta} y_{j,\beta} = \sum_{j=1}^m L_{y_j}(l_j).
\]
That is, vector $(c_{j,\beta})_{\beta}$ contains the coefficients of polynomials $l_j$ expressed in the monomial basis.
Similarly, the weak dynamics constraints of \eqref{eq:weakSwitchedAlternative} need only be satisfied for countably many polynomial test functions, since the measures are supported on compact subsets of $\R^{n+1}$.
Hence the weak dynamics defines a linear constraint between moments of the form
\begin{equation}\label{eq:bbeta}
\sum_{j=1}^m \sum_{\beta \in \N^{n+1}} a_{j,\alpha,\beta} y_{j,\beta}  = b_\alpha = v_\alpha(T,x_T) - v_\alpha(0,x_0)
\end{equation}
where coefficients $a_{j,\alpha,\beta}$ can  be deduced by identification from
\begin{equation}\label{eq:momLinConstr}
\sum_{\beta \in \N^{n+1}} a_{j,\alpha,\beta} y_{j,\beta} = L_{y_j}(\mathcal{L}'_j v_\alpha), \:\:\alpha \in \N^{n+1}.
\end{equation}
Finally, the only nonlinear constraints are the convex LMI constraints for measure representativeness, 
i.e., constraints on $y$ such that \eqref{mom} holds for all $j=1,\ldots,m$ and $\beta \in\N^{n+1}$. 
Indeed, it follows from Putinar's theorem \cite[Theorem 3.8]{lasserre} that 
the sequence of moments $y$ has a representing measure defined on
a set satisfying Assumption \ref{th:polynomialDataAssumption} if and only if $M_d(g_i\,y) \succeq 0$ for all $d \in \N$ and for all polynomials $g_i$ defining the set, $i=0,1,\ldots,n_K$, 
and where $A \succeq 0$ stands for matrix $A$ positive semidefinite, i.e., with nonnegative real eigenvalues. 

This leads to the problem:
\begin{equation}\label{eq:momentpb}
\begin{array}{rcll}
p^*_\infty & = & \displaystyle \inf_y &  \displaystyle \sum_{j=1}^m  \sum_{\beta \in \N^{n+1}} c_{j,\beta} y_{j,\beta}\\
&& \mathrm{s.t.} &  \displaystyle \sum_{j=1}^m \sum_{\beta \in \N^{n+1}} a_{j,\alpha,\beta} y_{j,\beta}  = b_\alpha, \:\alpha \in \N^{n+1},\\
&&& M_d(g_i\, y_j) \succeq 0, \: i=0\ldots n_K, \; j = 1 \ldots m,\\
& & & \qquad \qquad \qquad \; d \in \N.
\end{array}
\end{equation}

\begin{theorem}\label{th1}
Measure LP \eqref{eq:weakSwitchedAlternative} and infinite-dimensional LMI problem \eqref{eq:momentpb}
share the same optimum:
\[
    p^* = p^*_\infty.
\]
\end{theorem}
For the rest of the paper, we will therefore use $p^*$ to denote the cost of measure LP \eqref{eq:weakSwitchedAlternative} or LMI problem \eqref{eq:momentpb} indifferently.

\subsection{Moment LMI hierarchy}
\label{solvelp}

The final step to reach a tractable problem is relatively obvious: infinite-dimensional LMI problem \eqref{eq:momentpb} is truncated to its first few moments.

To streamline exposition, first notice in LMI \eqref{eq:momentpb} that $M_{d+1}( \cdot) \succeq 0$ implies $M_{d}( \cdot) \succeq 0$, such that when truncated, only the constraints of highest order must be taken into account.
Now, let $d_0 \in \N$ be the smallest integer such that all Lagrangians, dynamics and constraint monomials belong to $\N^{n+1}_{2 d_0}$.  This is the degree of the so-called \emph{first relaxation}. For any relaxation order $d \geq d_0$, the decision variable is now the vector $( y_{j,\alpha})_\alpha$ with $\alpha \in \N_{2 d}^{n+1}$, made of the first $ \left( \begin{smallmatrix} n+2d+1 \\ n+1 \end{smallmatrix} \right)$ moments of each measure $\mu_j$.  Then, define  the index set
\[
  \bar{\N}^{n+1}_{2d} := \left\lbrace \alpha \in \N^{n+1}: \, \deg \mathcal{L}'_j v_\alpha \leq 2 d, j = 1,\ldots, m \right\rbrace,
\]

viz., the set of monomials for which test functions of the form \eqref{eq:betaTestFun} lead to linear constraints of appropriate degree. By assumption, this set is finite and not empty -- the constant monomial always being a member.

Then, the LMI relaxation of order $d$ is given by
\begin{equation}\label{eq:momentrelax}
\begin{array}{rcll}
p^*_d & = &  \inf\limits_y &  \displaystyle \sum_{j=1}^m  \sum_{\beta \in \N^{n+1}_{2d}} c_{j,\beta} y_{j,\beta} \\
&& \text{s.t.} &  \displaystyle \sum_{j=1}^m \sum_{\beta \in \N^{n+1}_{2d}} a_{j,\alpha,\beta} y_{j,\beta}  = b_\alpha, \:\alpha \in \bar{\N}^{n+1}_{2d} \\
&&& M_d(g_i\, y_j) \succeq 0, \: i=0,\ldots , n_K, \; j = 1, \ldots, m.
\end{array}
\end{equation}

Notice that for each relaxation, we get a standard LMI problem that can be solved numerically by off-the-shelf software. In addition, the relaxations converge asymptotically to the cost of the moment LP:

\begin{theorem}\label{th2}
For increasing relaxation orders $d$,
the hierarchy of finite-dimensional LMI relaxations (\ref{eq:momentrelax}) yields an asymptotically converging
monotonically non-decreasing sequence of lower bounds on the values of measure LP
\eqref{eq:weakSwitchedAlternative} and infinite-dimensional LMI problem \eqref{eq:momentpb} , i.e.
\[
p^*_{d_0} \leq p^*_{d_0+1} \leq \cdots \leq p^*_{\infty}  = p^*.
\]
\end{theorem}

\begin{proof}
By construction, observe that $j > i$ implies $p^*_{d_0+j} \geq p^*_{d_0+i} $, viz.,
the sequence $p^*_d$ is monotonically non-decreasing. Asymptotic convergence to $p^*$
follows from \cite[Theorem 3.8]{lasserre} as in the proof of Theorem \ref{th1}.
\end{proof}

Therefore, by solving the truncated problem for ever greater relaxation orders, we obtain a monotonically non-decreasing sequence
of lower bounds to the true optimal cost.

\subsection{SOS LMI hierarchy}
\label{sec:SOS}

As for the measure LP of \S\ref{sec:measureLP}, the moment LMI relaxations detailed in the previous section possess a conic dual. In this section, we show that this dual problem can be interpreted as a polynomial SOS strengthening of the dual outlined in \S\ref{sec:duality}.
The exact form of the dual problem is an essential aspect of the numerical method, since it will be solved implicitly whenever primal-dual interior-point algorithms are used for solving the moment LMI hierarchy of \S\ref{solvelp}.

Let ${\mathbb S}^n$ be the space of symmetric $n \times n$ real matrices. One can show (see, e.g., \cite[\S{}C]{lasserre}) that, for $A,B \in {\mathbb S}^n$, $\langle A, B \rangle := \mathrm{trace} \, AB$ is a duality bracket ${\mathbb S}^n \times {\mathbb S}^n \mapsto \R$, and that $\left\lbrace A \in {\mathbb S}^n: \, A \succeq 0 \right\rbrace$ defines a convex cone of ${\mathbb S}^n$.
In problem \eqref{eq:momentrelax}, let us define the matrices $A_{i,\beta} \in {\mathbb S}^{\left( \begin{smallmatrix} n+d+1 \\ n+1 \end{smallmatrix} \right)}$ satisfying the identity
\[
M_d(g_i\:y) = \sum_\beta A_{i,\beta} y_{\beta}
\]
for every sequence $(y_{\beta})_{\beta}$ and $i=0,1,\ldots,n_K$.

\begin{proposition}
The conic dual of moment LMI problem \eqref{eq:momentrelax} is given by the SOS LMI problem
\begin{equation}\label{eq:dualmomentrelax}
\begin{array}{ll}
\sup\limits_{z, Z} &  \displaystyle \sum_{\alpha \in \bar{\N}^{n+1}_{2d}} b_{\alpha} z_{\alpha} \\
\text{s.t.} &  \displaystyle \sum_{\alpha \in \bar{\N}^{n+1}_{2d}} a_{j,\alpha,\beta} z_\alpha + \sum_{i=0}^{n_K} \langle A_{i,\beta}, Z_{i,j} \rangle =  c_{j,\beta}, \: \beta \in \N^{n+1}_{2 d},\\
& Z_{i,j} \succeq 0,  \quad i = 0, 1,\ldots , n_K, \; j = 1, \ldots , m. \\
\end{array}
\end{equation}

\end{proposition}

\begin{proof}
Replacing the equality constraints in \eqref{eq:momentrelax} as two inequalities, it is easy to see that the moment relaxation can be written as an instance of LP \eqref{dual}, whose dual is given symbolically by \eqref{eq:primal}. Working out the details leads to the desired result, using for semi-definite constraints the duality bracket as explained earlier.
\end{proof}

The relationship between \eqref{eq:dualmomentrelax} and \eqref{eq:dualExplicit} might not be obvious at a first glance. Denote by $\Sigma[z]$ the subset of $\R[z]$ that can be expressed as a finite sum of squares of polynomials. Then a standard interpretation (see, e.g., \cite{lasserre}) of 
\eqref{eq:dualmomentrelax} in terms of such objects is given by the next proposition.

\begin{proposition} \label{th:polyStrength}
LMI problem  \eqref{eq:dualmomentrelax} can be stated as the following polynomial SOS
strengthening of problem \eqref{eq:dualExplicit}:
\begin{equation}\label{eq:SOSstrength}
\begin{array}{ll}
\sup &  v(T,x_T) - v(0,x_0) \\
\text{s.t.} & \displaystyle l_j - \mathcal{L}'_j v = \sum_{i=0}^{n_K} g_i s_{i,j}, \quad j = 1, \ldots, m\\
\end{array}
\end{equation}
where the maximization is w.r.t. the vector of coefficients $z$ of polynomial
$v(t,x) = \sum_{\alpha \in \bar{\N}^{n+1}_{2d}} z_\alpha v_\alpha(t,x)$
and the vectors of coefficients of polynomials
$s_{i,j} \in \Sigma [t,x], \; \deg g_i \, s_{i,j} \leq 2d$.
\end{proposition}

\begin{proof}
By \eqref{eq:bbeta}, the cost of \eqref{eq:SOSstrength} is equivalent to \eqref{eq:dualmomentrelax}. Then multiply each scalar constraint (indexed by $\alpha$) by $v_\alpha$ as given by \eqref{eq:betaTestFun}, and sum them up. By definition, $\sum_\alpha c_{j,\alpha} v_\alpha = l_j$, and similarly, $\sum_{\beta} a_{j,\alpha,\beta} z_\alpha v_\alpha = \mathcal{L}'_j v_\alpha$. The conversion from the semi-definite terms to SOS exploits their well-known relationship (e.g., \cite[\S{}4.2]{lasserre}) to obtain the desired result, by definition \eqref{eq:locMomMat} of the localizing matrices.
\end{proof}

Prop. \ref{th:polyStrength} specifies in which sense ``polynomial SOS strengthenings'' must be understood: positivity constraints of \eqref{eq:dualExplicit} are enforced by SOS certificates, and the decision variable of \eqref{eq:dualExplicit} is now limited to polynomials of appropriate degrees.

Finally, the following result states that no numerical troubles
are expected when using classical interior-point algorithms
on the primal-dual LMI pair (\ref{eq:momentrelax}-\ref{eq:dualmomentrelax}).

\begin{proposition}
The infimum in primal LMI problem (\ref{eq:momentrelax}) is equal to
the supremum in dual LMI problem (\ref{eq:dualmomentrelax}),
i.e., there is no duality gap.
\end{proposition}

\begin{proof}
By Assumption \ref{th:polynomialDataAssumption}, one of the polynomials $g_i$ in the description of set $K$
enforces a ball constraint. The corresponding localizing constraint $M_d(g_i\, y_j) \succeq 0$ then implies
that the vector of moments $y$ is bounded in LMI problem (\ref{eq:momentrelax}). 
Then to prove the absence of duality gap, we use
the same arguments as in the proof of Theorem 4 in Appendix D of \cite{roa}.
\end{proof}

\subsection{ Computational  complexity}
\label{sec:complexity}
This section details rough estimates on the numerical gains expected by employing modal occupation measures in LMI relaxations as opposed to the more generic method developed in \cite{sicon}.

Following the standing assumptions, the computational complexity of solving an LMI relaxation of order $d$ with 
$\bar{p}$ occupation measures supported on a space of dimension $n+m+1$  (time, states, controls) is dominated by $\bar{p}$ LMI constraints of size $\bar{m}$ (the moment matrices of the occupation measures) in $\bar{n}$ variables (the moment vectors of the occupation measures), with
\[
\bar{n} = \begin{pmatrix}
n+m+2d+1 \\ n+m+1
\end{pmatrix}, \quad
\bar{m} = \begin{pmatrix}
n+m+d+1 \\n+m+1
\end{pmatrix}.
\]
A standard primal-dual interior-point algorithm to solve this LMI at given relative accuracy $\epsilon > 0$ (duality gap threshold) requires a number of iterations (Newton steps) growing as $\mathcal{O}(\bar{p}^{\frac{1}{2}} \bar{m}^{\frac{1}{2}}) \log \epsilon $, see, e.g., \cite[Section 6.6.3]{BenTal2001lecture}. In the real model of computation (for which each addition, subtraction, multiplication, division of real numbers has unit cost), each Newton iteration requires no more than $\mathcal{O}(\bar{n}^3)+\mathcal{O}(\bar{p}\bar{n}^2\bar{m}^2)+\mathcal{O}(\bar{p}\bar{n}\bar{m}^3)$ operations. When solving a hierarchy of simple LMI relaxations as described above, the accuracy $\epsilon$ is fixed, the number of states $n$ and controls $m$ are fixed the number $\bar{p}$ of measures is fixed, and the relaxation order $d$ varies, such that $\mathcal{O}(\bar{n}) = \mathcal{O}(\bar{m}) = \mathcal{O}(d^{4(n+m+1)})$ and the dominating term in the complexity estimate grows in $\mathcal{O}(\bar{p}^{\frac{3}{2}}d^{\frac{9}{2}(n+m+1)})$, which clearly
shows a strong dependence on the number of state and control variables.

Therefore, an LMI relaxation of an optimal switching control  problem with $n$ states and $m$ controls will be solved in $\mathcal{O}(d^{\frac{9}{2} \, (n+m+1)})$ operations when solved with the general formulation of \cite{sicon} and $\mathcal{O}( m^{\frac{3}{2}} \,d^{\frac{9}{2} \, (n+1)})$ when solved with structured formulation \eqref{eq:weakSwitchedAlternative}. That is a $\frac{n+1}{n+m+1}$ reduction of the polynomial rate at which the CPU time grows with relaxation orders. This back-of-the-envelope estimation is slightly conservative in practice, as shown in \S\ref{sec:doubleint}. Note that the same analysis as performed with the reported computation times of \cite{Bonnard2014NMR} or \cite{Sager2014electricCar} reveals a similar underestimation of actual computation gains by applying the method described in this paper. As seen in these examples, those gains are substantial and justify the practical use of the modal occupation measures for switched systems. 

\section{Trajectory extraction}
\label{sec:extract}

This section presents a method for extracting approximate optimal trajectories from moment data. 

As presented so far, the method provides a converging sequence of lower bounds to the value of the control problem. This already provides a valuable tool to the user, allowing to certify as \emph{globally} optimal local solutions found by other means such as a local solver or heuristics -- or even reject unsound candidate solutions, see \S\ref{sec:LQRex}. In addition, comparing lower bounds with the value of an admissible solution, when they do not match, allows to easily quantify the maximal gain that may be achieved by investing more numerical efforts for that particular problem.

As useful as these may prove, it seems intuitive that much more than the cost could be extracted from moment data, especially at sufficiently high relaxation orders. In this section, we show that optimal controls and trajectories can indeed be inferred. The focus here is not on obtaining very precise estimates of the optimal solution, as this would require an impractically large number of test functions in the weak ODE constraints. In addition, estimates of arbitrary high precision can already be obtained routinely from local optimal control solvers (see, e.g., \texttt{BOCOP} \cite{bocop}), \emph{if suitably initialized}. Therefore, what is only needed from the moment data is a sensible approximation of the globally optimal solution(s), to be then simply refined by local methods.

Consider the trajectory $x(t)$ which must be reconstructed on $[0,T]$ from partial moment sequence $\left\lbrace L_{y^j}(t^\alpha x^\beta) \right\rbrace_{\alpha+ \lvert \beta \rvert \leq 2r}$, $j=1 \ldots m$. As a consequence of \eqref{th:noGapVinter}, see \cite[Cor.~1.4]{Vinter1993ConvexDuality}, optimal solutions $\lbrace \mu_j \rbrace$ are supported on the graph of optimal trajectories. Therefore, one possible strategy for reconstructing $x(t)$ is to construct an approximate support on a discretization of $[0,T] \times X$ from the partial moment data. We propose a specialization for switched systems - needed for recovering the modal duty cycles -- of the method developed in \cite{claeys2014reconstruction}. As in this reference, we assume that the optimal pair $(u^*(t),x^*(t))$ is unique for ease of exposition. We also assume for the same reason that the state space is unidimensional, as the outlined procedure could be trivially repeated for each component of a multidimensional state space.

Fix a mesh resolution $\varepsilon$. Denote then by $X_\varepsilon$ a mesh supported on finitely many points, such that for all $x \in X$, there exists $x_i \in X_\varepsilon$ such that $\lvert x_i - x \rvert \leq \varepsilon$.
Let $T_\varepsilon$ denote the analogous discretization of $[0,T]$. It is proposed in \cite[Th.~3]{claeys2014reconstruction} to solve the program
\begin{equation}\label{eq:approxLP}
\begin{aligned}
\min_{\tilde{\mu}_j}  \; & \lVert L_{y^j}(t^\alpha x^\beta) - \langle t^\alpha x^\beta ,\tilde{\mu}_j \rangle \rVert \\
\text{s.t.} \; & \tilde{\mu}_j \in \PositiveMeasures ( T_\varepsilon \times X_\varepsilon ),
\end{aligned}
\end{equation}
to obtain the best discrete approximation $\tilde{\mu}_j$ of the partial moment sequence $\lbrace y^j \rbrace$. Indeed, this can be recast as a simple finite-dimensional, linear program, since measures in $\PositiveMeasures ( T_\varepsilon \times X_\varepsilon )$ can be parametrized by a discrete set of weights, each assigned to an atom supported on $T_\varepsilon \times X_\varepsilon$. The approximate trajectory way points are simply deduced from the support of the optimal solution of \eqref{eq:approxLP}, viz., from the non-zero atomic weights.

This simple procedure could be repeated for each modal measure $\mu_j$ in order to reconstruct the full trajectory. Active modes along the trajectory are then deduced by inspecting the relative modal weights at each atom. However, this naive approach proves to be very imprecise in practice, since modes are deduced from solutions of independent optimization problems (one for each measure). In particular, the sum of all modal measures must disintegrate as (as plugging test functions of time only in \eqref{eq:weak} reveals):
\begin{equation}\label{eq:disintegration}
 \sum_{j=1}^m \mu_j = \zeta(\diff x | t) \, \diff t,
\end{equation}
where $\zeta$ is a probability measure on $X$ define almost-everywhere on $[0,T]$. However, the discrete version of this property is not guaranteed with independent solving of \eqref{eq:approxLP}, so that duty cycles can hardly be reconstructed. This observation leads directly to the following alternative program:
\begin{equation}\label{eq:approxLPbis}
\begin{aligned}
\min_{\tilde{\mu}_1, \ldots, \tilde{\mu}_m }  \; & \sum_{j=1}^m L_{y^j}(t^\alpha x^\beta) - \langle t^\alpha x^\beta ,\tilde{\mu}_j \rangle \\
\text{s.t.} \; & \sum_{j=1}^m \sum_{k=1}^{\lvert X_\epsilon \rvert} \tilde{\mu}_j \lbrace (t_i,  x_k ) \rbrace = (\Delta t)_i, \quad i=1,\ldots,\lvert T_\epsilon \rvert,\\
& \tilde{\mu}_j \in \PositiveMeasures ( T_\varepsilon \times X_\varepsilon ),
\end{aligned}
\end{equation}
where $(\Delta t)_i$ is the length of the $i$th cell, containing $t_i \in T_\epsilon$, of a given partition of $[0,T]$. In \eqref{eq:approxLPbis}, we have done nothing else that to couple the modal reconstruction by constraining the discretization of $\zeta(\diff x | t_i)$ to be a probability measure. When the solution of the optimal control problem is unique, it is then possible to approximate the value of $x^*(t)$ at the time nodes via
\[
x^*(t_i) \approx \frac{\sum_{j=1}^m \sum_{k=1}^{\lvert X_\epsilon \rvert} \tilde{\mu}_j \lbrace( t_i, x_k ) \rbrace \, x_k}{(\Delta t)_i},
\]
and the duty cycles via
\[
d_j^*(t_k) \approx \frac{\sum_{k=1}^{\lvert X_\epsilon \rvert} \tilde{\mu}_j \lbrace ( t_i, x_k ) \rbrace }{(\Delta t)_i}.
\]
To show that a relaxation has converged, we use those way-points to hot start a numerical method, see the examples in \S\ref{sec:examples}.
If the cost of the local method agrees with some tolerance to the relaxation cost, we then terminate the moment hierarchy and validate the local solution as globally optimal.
In practice, and in accordance with the analysis of \S\ref{sec:complexity}, problems with state-space smaller than 6 can be usually solved with a globality gap of a few percent at most.
The method may still be applied to larger problems; For some instances, one may still solve the problem, whereas on some others, the method simply provide a lower bound on the cost.

\section{Extensions}
\label{sec:extensions}

Several standard extensions of our results are now presented.

\subsection{Free/distributed initial state}

The approach can easily take into account a free initial state and/or time, by introducing an initial occupation measure
 $\mu_0 \in \PositiveMeasures(\{0\}\times X_0)$, with $X_0 \subset \R^n$ 
a given compact set, such that for an admissible starting point $(t_0,x(t_0))$,
\begin{equation}\label{eq:weakInitial}
\langle v, \mu_0 \rangle = v(t_0,x(t_0)),
\end{equation}
for all continuous test functions $v(t,x)$ of time and space. That is, in weak problem \eqref{eq:weakSwitchedAlternative}, one replaces some of the boundary conditions by injecting \eqref{eq:weakInitial} and making $\mu_0$ an additional decision variable. In dual \eqref{eq:dualExplicit}, this adds a constraint on the initial value and modifies the cost.

Similarly, a terminal occupation measure $\mu_T$ can be introduced for free terminal states and/or time. Note that injecting  both \eqref{eq:weakInitial} and its terminal counterpart in \eqref{eq:weakSwitchedAlternative} requires the introduction of an additional affine constraint to exclude trivial solutions, e.g., $\langle 1, \mu_0 \rangle = 1$ so that $\mu_0$ is a probability measure.
In dual problem \eqref{eq:dualExplicit}, this introduces an additional decision variable.

More interestingly, the initial time may be fixed w.l.g. to $t_0 = 0$, but only the spatial probabilistic distribution of initial states is known. Let $\xi_0(\diff x) \in \mathcal{P}(X_0)$ be the measure whose law describes such a distribution. Then, the additional constraints for \eqref{eq:weakSwitchedAlternative} are
\[
\langle v, \mu_0 \rangle = \int_{X_0} \!\! v(0,x) \, \xi_0(\diff x).
\]
As remarked in \cite{sicon}, this changes the interpretation of LP \eqref{eq:weakSwitchedAlternative} to the minimization of the expected value of the cost given the initial distribution. See also \cite{roa} for the Liouville interpretation of the LP as transporting measure $\xi_0$ along the optimal flow.

\subsection{Additional constraints}
\label{sec:extraconstraints}

In measure LP \eqref{eq:weakSwitchedAlternative}, each modal occupation measure is supported on the same set $K$. This simply translates the fact that state constraints are mode-independent. However, nothing prevents the use of specific mode constraints by defining $m$ sets $\lbrace K_j \rbrace_{j=1,\ldots,m}$. In unstructured LP \eqref{eq:weak}, these would be specified by constraining the support of $\mu$ to $\prod_{j=1}^m \lbrace e_j \rbrace \times K_j $, where $e_j$ is the unit vector in $\R^m$ whose $j$-th entry is $1$. Then, in the proof of Th.~\ref{th:noGapModal}, for direction $p^* \leq p_{W}^*$, it is easy to see that this implies that each $\mu_j$ is supported on $K_j$. For the reverse direction, one can define w.l.g. the support of $\bar{\mu}$ to be $K:=\cup_{j=1}^m K_j$ by extending that of each $\mu_j$ to $K$ and setting $\mu_j(K \setminus K_j)=0$. Then $\mu$ has the expected support.

A second class of constraints that can be easily accounted for by the method are a finite number of integral constraints of the form
\begin{equation}\label{eq:integralConstr}
\int_0^T \!\! h_{\sigma(t)}^k(t,x(t)) \, \diff t \leq e_k, \qquad k=1,\ldots,n_i
\end{equation}
with given $h_j^k \in \R[t,x]$, $j=1,\ldots,m$ and $e_k \in \R$, $k=1,\ldots,n_i$. Indeed, rewriting the constraint in the equivalent formalism of \eqref{discrete} and using modal occupation measures as in Def.~\ref{th:modalOccupMeas}, \eqref{eq:integralConstr} becomes
\begin{equation}\label{eq:integralConstrModal}
\sum_{j=1}^{m} \langle h^k_j, \mu_j \rangle \leq e_k.
\end{equation}
That is, in moment problem \eqref{eq:momentrelax}, whose first relaxation order may need to be updated in function of the degrees of each $h_j^k$, each of the $n_i$ integral constraints simply appears as an additional linear constraint. The interest of treating these integral constraints explicitly (instead of rewriting them as additional states to fit in the rigid formalism of \eqref{ocp}) is obvious from \S\ref{sec:complexity} in terms of computational load. Conversely,  some translational-invariant problems might strongly benefit in replacing selected states by integral constraints, see, e.g., \cite{Sager2014electricCar}.

\subsection{Wider class of dynamical systems}\label{sec:wider}

First note that additional drift terms, viz., dynamics of the form $\dot{x} = f_0 + \sum_{j=1}^m f_j \, u_j$, fit easily within the framework of problem \eqref{ocp} by adding the drift term to each mode dynamics.

Secondly, in a related fashion, affine-in-the-control problems with polytopic control sets
 more general that the box described in (\ref{eq:simplex}) 
can also be handled by the formalism presented in this paper. Indeed, the problem is equivalent to controlling (with relaxed objects -- see section \S\ref{sec:measureLP}) the system with 
modes associated with vertices of the control polytope. 

Finally, dynamics of the form $f_j(t,x,u_j)$, parametrized by a mode-dependent measurable control $u_j(t)$, can be incorporated into the presented formalism by combining the approach in this paper with that of \cite{sicon}. However, mode-dependent states, where the size of the state-space varies for each given mode change and/or whose meaning is mode-dependent, must be incorporated by extending those states to all modes with a null vector field. It is an open question whether a more elegant approach can be incorporated into the framework of this paper.

\subsection{Open-loop versus closed-loop}
In problem (\ref{ocp}), the control signal is the switching sequence $\sigma(t)$ which is
a function of time: this is an open-loop control, similarly to what was proposed in \cite{impulse}
for impulsive control design. One could also want to
constrain the switching sequence to be an explicit or implicit function of the state,
i.e., $\sigma(x(t))$, a closed-loop control signal. In this case, each occupation measure, explicitly depending on time, state and control, may be disintegrated as
$\diff\mu(t,x,u)=\xi(\diff t\:|\:t,u)\omega(\diff u\:|\:t)\diff t$, and we should follow the framework
described originally in \cite{sicon}. 

\subsection{Switching and impulsive control}
We may also combine switching control and impulsive control
if we extend the system dynamics to
\[
\diff x(t) = \sum_{k=1}^m f_k(x(t)) \, u(t) \, \diff t + \sum_{j=1}^{p} g_j(t) \nu_j(\diff t),
\]
which must be understood in a weak sense. Here, $g_j$ are given continuous vector functions of time
and $\nu_j$ are signed measures to be found, jointly with
the switching signal $u(t)$. Whereas modal occupation measures are
restricted to be absolutely continuous w.r.t.
the Lebesgue measure of time, impulsive control measures $\nu_j$ can
concentrate in time. For example, for a dynamical system
$\diff x(t) = g(t)\nu( \diff t)$, a Dirac measure $\nu( \diff t)=\delta_{s}$ 
enforces at time $t=s$ a state jump $x^+(s) = x^-(s) + g(s)$.
In this case, to avoid trivial solutions, the objective function
should penalize the total variation of the impulsive control measures,
see \cite{impulse}.

\section{Examples}\label{sec:examples}

In this section, we show several examples to illustrate the method.
The first example illustrates the fact that
the infimum is not attained even in the simplest instances of
optimal control problem (\ref{ocp}).
The second example is the constrained double integrator example of \cite{sicon}. 
The final two examples are taken from \cite{vasudevan2014consistent}, where they are solved by a local optimization method dedicated for switched systems. These last examples highlight how our method can complement more traditional approaches.

All moment relaxations were built with the \texttt{GloptiPoly} toolbox \cite{GloptiPoly} and solved via the \texttt{Mosek} semi-definite solver \cite{mosek}, interfaced through \texttt{Yalmip} \cite{yalmip}.

\subsection{Chattering}\label{sec:exScalar}

Consider the scalar $(n=1)$ optimal control problem (\ref{ocp}):
\[
\begin{array}{rcll}
p^* & = & \inf & \displaystyle\int_0^1 x^2(t) \diff t \\
&& \mathrm{s.t.} & \dot{x}(t) = a_{\sigma(t)}x(t),\\
&&& x(0) = \frac{1}{2},\\
&&& x(t) \in [-1,\:1], \quad\forall t \in [0,1]
\end{array}
\]
where the infimum is w.r.t. a switching sequence
$\sigma : [0,1] \mapsto \{1,\:2\}$ and
\[
a_1 := -1, \quad a_2 := 1.
\]
In Table \ref{ex1table} we report the lower bounds $p^*_d$ on
the optimal value $p^*$ obtained by solving LMI relaxations (\ref{eq:momentrelax}) of
increasing orders $d$, rounded to 5 significant digits.
We also indicate the number of variables $\bar{n}$
(i.e., total number of moments) of each LMI problem, as well as
the zeroth order moment of each occupation measure
(recall that these are approximations of the time spent
on each mode). We observe that the values of the
lower bounds and the masses stabilize quickly.

\begin{table}[h!]
\centering
\begin{tabular}{c|c|c|cc}
$d$ & $p^*_d$ & $\bar{n}$ & $y_{1,0}$ & $y_{2,0}$ \\ \hline
1 & $-5.9672\cdot10^{-9}$ & 18 & 0.74056 & 0.25944 \\
2 & $4.1001\cdot10^{-2}$ & 45 & 0.75170 & 0.24830 \\
3 & $4.1649\cdot10^{-2}$ & 84 & 0.74632 & 0.25368 \\
4 & $4.1666\cdot10^{-2}$ & 135 & 0.74918 & 0.25082 \\
5 & $4.1667\cdot10^{-2}$ & 198 & 0.74974 & 0.25026 \\
6 & $4.1667\cdot10^{-2}$ & 273 & 0.74990 & 0.25010 \\
7 & $4.1667\cdot10^{-2}$ & 360 & 0.74996 & 0.25004 \\
\end{tabular}
\caption{Lower bounds $p^*_d$ on the optimal value $p^*$ obtained
by solving LMI relaxations of increasing orders $d$;
$\bar{n}$ is the number of variables in the LMI problem;
$y_{j,0}$ is the approximate time spent on each mode $j=1,2$.\label{ex1table}}
\end{table}

In this simple case, it is easy to obtain analytically the
optimal switching sequence: it consists of driving the state
from $x(0)=\frac{1}{2}$ to $x(\frac{1}{2})=0$ with the first
mode, i.e., $u_1(t)=1,u_2(t)=0$ for $t \in [0,\frac{1}{2}[$, and
then chattering between the first and second mode with
equal proportion so as to keep $x(t)=0$, i.e., $u_1(t)=\frac{1}{2},u_2(t)=\frac{1}{2}$
for $t \in ]\frac{1}{2},1]$. It follows that the infimum
is equal to
\[
p^* = \int_0^{1/2} \left(\frac{1}{2}-t\right)^2 \diff t = \frac{1}{24} \approx 4.1667\cdot10^{-2}.
\]
Because of chattering, the infimum in problem (\ref{ocp})
is not attained by an admissible
switching sequence. It is however
attained in the convexified problem (\ref{convex}).

The optimal moments can be obtained analytically
\[
y_{1,{\alpha}} = \int_0^{\frac{1}{2}} t^{\alpha} \diff t + \frac{1}{2} \int_{\frac{1}{2}}^1
t^{\alpha} \diff t = \frac{2+2^{-\alpha}}{4+4\alpha},
\]
\[
y_{2,{\alpha}} = \frac{1}{2} \int_{\frac{1}{2}}^1 t^{\alpha} \diff t
= \frac{2-2^{-\alpha}}{4+4\alpha}
\]
and they can be compared with the following moment vectors
obtained numerically at the 7th LMI relaxation:
\[
\begin{array}{rcl}
\mathrm{computed}\: y_1 & = & \left(0.74996 \:\: 0.31246 \:\: 0.18746 \:\: 0.13277 \cdots \right), \\
\mathrm{exact}\:y_1 & = & \left(0.75000 \:\: 0.31250 \:\: 0.18750 \:\: 0.13281 \cdots \right), \\
\mathrm{computed}\:y_2 & = & \left(0.25004 \:\: 0.18754 \:\: 0.14588 \:\: 0.11723 \cdots \right), \\
\mathrm{exact}\:y_2 & = & \left(0.25000 \:\: 0.18750 \:\: 0.14583 \:\: 0.11719 \cdots \right).
\end{array}
\]
We observe that the computed moments  closely
match the exact optimal moments, so that the approximate
control law extracted from the computed moments will be
almost optimal.

\subsection{Double integrator}\label{sec:doubleint}

We revisit the double integrator example with state constraint
studied in \cite{sicon}, formulated as the following
optimal switching problem:
\[
\begin{array}{rcll}
p^* & = & \inf & T \\
&& \mathrm{s.t.} & \dot{x}(t) = f_{\sigma(t)}(x(t)),\\
&&& x(0) = (1,\:1), \quad x(T) = (0,\:0)\\
&&& x_2(t) \geq -1, \quad\forall t \in [0,T]
\end{array}
\]
where the infimum is w.r.t. a switching sequence
$\sigma : [0,T] \mapsto \{1,\:2\}$ with free terminal time $T\geq 0$
and affine dynamics
\[
f_1 := \left(\begin{array}{c}x_2\\-1\end{array}\right), \quad
f_2 := \left(\begin{array}{c}x_2\\1\end{array}\right).
\]
We know from \cite{sicon} that the optimal sequence consists of starting with
mode $1$, i.e., $u_1(t)=1$, $u_2(t)=0$ for $t \in [0,2]$,
then chattering with equal proportion between mode $1$ and $2$,
i.e., $u_1(t)=u_2(t)=\frac{1}{2}$ for $t \in [2,\frac{5}{2}]$
and then eventually driving the state to the origin
with mode $2$, i.e., $u_1(t)=0$, $u_2(t)=1$ for $t \in [\frac{5}{2},\frac{7}{2}]$.
Here too the infimum $p^*=\frac{7}{2}$ is not attained for
problem (\ref{ocp}), whereas it is attained with the
above controls for problem (\ref{convex}).

In Table \ref{ex2table} we report the lower bounds $p^*_d$ on
the optimal value $p^*$ obtained by solving LMI relaxations (\ref{eq:momentrelax}) of
increasing orders $d$, rounded to 5 significant digits.
We also indicate the number of variables $\bar{n}$
(i.e., total number of moments) of each LMI problem, as well as
the zeroth order moment of each occupation measure
(recall that these are approximations of the time spent
on each mode). We observe that the values of the
lower bounds and the masses stabilize quickly
to the optimal values $p^*=\frac{7}{2}$, $y_{1,0}=\frac{5}{2}$,
$y_{2,0}=\frac{5}{4}$.

\begin{table}[h!]
\centering
\begin{tabular}{c|c|c|cc}
$d$ & $p^*_d$ & $\bar{n}$ & $y_{1,0}$ & $y_{2,0}$ \\ \hline
1 & $2.5000$ & 30 & 1.7500 & 0.75000 \\
2 & $3.2015$ & 105 & 2.1008 & 1.1008 \\
3 & $3.4876$ & 252 & 2.2438 & 1.2438 \\
4 & $3.4967$ & 495 & 2.2484 & 1.2484 \\
5 & $3.4988$ & 858 & 2.2494 & 1.2494 \\
6 & $3.4993$ & 1365 & 2.2496 & 1.2497 \\
7 & $3.4996$ & 2040 & 2.2498 & 1.2498 \\
\end{tabular}
\caption{Lower bounds $p^*_d$ on the optimal value $p^*$ obtained
by solving LMI relaxations of increasing orders $d$;
$\bar{n}$ is the number of variables in the LMI problem;
$y_{j,0}$ is the approximate time spent on each mode $j=1,2$.\label{ex2table}}
\end{table}

{We carried out a comparison of computational times required for solving the unstructured generic
LMI relaxations of \cite{sicon} versus their structured modal counterpart, for relaxation orders $1$ to $7$.
The computational gains induced by using modal occupation measures follow the expected exponential
improvement, with an exponent of about $1.7$, better than the $\frac{4}{3}$ predicted by the asymptotic analysis of \S\ref{sec:complexity}.}


\subsection{Switched LQR}\label{sec:LQRex}

We consider the problem
\[
\begin{aligned}
p^* & = & \inf & \int_0^T \! 0.01 \, u^2(t) \, \diff t + \sum_{i=1}^3 (x_i(T)-1)^2 \\
&& \mathrm{s.t.}\; & \dot{x}(t) = A \, x(t) + B_{\sigma(t)} \, u(t),\\
&&& x(0) = (0,0,0),\\
&&& u(t) \in U = [-20,20],\\
&&& \sigma \in \lbrace 1, 2, 3 \rbrace,\\
&&& T = 2,
\end{aligned}
\]
with problem data
\[
\begin{array}{c}
A = \left(\begin{smallmatrix} 1.0979 & -0.0105 & 0.0167 \\ -0.0105 & 1.0481 & 0.0825 \\ 0.0167 & 0.0825 & 1.1540\end{smallmatrix}  \right), \\[1em]
B_1 = \left(\begin{smallmatrix} 0.9801 \\ -0.1987 \\ 0 \end{smallmatrix} \right) , \;
B_2 = \left(\begin{smallmatrix} 0.1743 \\ 0.8601 \\ -0.4794 \end{smallmatrix} \right) , \;
B_3 = \left( \begin{smallmatrix} 0.0952 \\ 0.4699 \\ 0.8776 \end{smallmatrix} \right)
\end{array}
\]
{and where the infimum is w.r.t. a switching sequence $\sigma(t)$ but 
also a classical control $u(t)$, consistently with the extension sketched in \S \ref{sec:wider}.}
The lower bounds and computation times are given in Tab.~\ref{tab:LQRex}. We also report the cost of a direct, local, optimal control methods as implemented in \texttt{BOCOP} \cite{bocop} via embedding \eqref{convex}, and hot started with the method of \S\ref{sec:extract}. As one can see, the costs are in close agreement, and we can numerically certify the local solutions as globally optimal. We also used \texttt{VSDP} \cite{vsdp} to strengthen the semi-definite bounds, so as to compute verified lower bounds on the cost of a relaxation via rigorous interval arithmetic. The lower bound for the third order relaxation could be certified to the significant figures presented in Tab.~\ref{tab:LQRex}. The disagreement with the lower cost reported in \cite{vasudevan2014consistent} is due to numerical round-off issues in the execution of their local method. In fact, integrating finely the solution reported in \cite{vasudevan2014consistent} leads the a candidate arc whose cost is \emph{above} our certified lower bound.

\begin{table}
\centering
\begin{tabular}{c|c|c}
$d$ & $p^*_d$ & $\bar{n}$  \\ \hline
1 & $0.071 \, 10^{-3}$ & 93 \\
2 & $1.823 \, 10^{-3}$ & 518 \\
3 & $1.829 \, 10^{-3}$ & 1806 \\
\ldots & & \\
\texttt{BOCOP} & $1.831 \, 10^{-3}$ &
\end{tabular}
\caption{Lower bounds $p^*_d$ and number of variables $\bar{n}$ of the LMI problem in function of relaxation order $d$ for the problem of \S\ref{sec:LQRex}. The value found by hot starting \texttt{BOCOP} with our method is also shown.
}
\label{tab:LQRex}
\end{table}

\subsection{Double tank problem}\label{sec:tankex}

We consider the double tank problem of \cite{vasudevan2014consistent}:
\[
\begin{aligned}
p^* & = & \inf & \int_0^T \! 2 (x_2(t)-3) \, \diff t \\
&& \mathrm{s.t.}\; & \dot{x}_1(t) = c_{\sigma(t)} - \sqrt{x_1(t)} \\
&&& \dot{x}_2(t) = \sqrt{x_1(t)} - \sqrt{x_2(t)},\\
&&& x(0) = (2,\:2)' \\
&&& \sigma \in \lbrace 1, 2 \rbrace, \\
&&& T = 10,
\end{aligned}
\]
with data $c_1 = 1, c_2 = 2$ { and where the infimum is w.r.t. a switching sequence $\sigma(t)$}.
By introducing the lifts $\ell_i = \sqrt{x_i}$, algebraically constrained as $\ell_i^2 = x_i$, $\ell_i \geq 0$, the problem is easily recast with polynomial data. Table \ref{tab:tankex} reports the relaxation costs of our method along with the cost of \texttt{BOCOP}, hot-started by the moment data. Again, agreement is reached quickly. Note that the cost is here strictly better than the one reported in \cite{vasudevan2014consistent}, which is stuck in a local but not global minimum. This adverse behavior is of course inherent to any local optimization scheme, and shows the interest of methods which allow to obtain decent starting point to initialize them. Obviously, such a pass with a local method is still necessary to obtain very precise numeric solutions, and our general framework allows to combine the best of the global and local worlds.

\begin{table}
\centering
\begin{tabular}{c|c|c}
$d$ & $p^*_d$ & $\bar{n}$  \\ \hline
1 & $0.0000$ & 62 \\
2 & $4.4886$ & 322 \\
3 & $4.7265$ & 1092 \\
4 & $4.7298$ & 2904 \\
5 & $4.7304$ & 6578 \\
\ldots & & \\
\texttt{BOCOP} & $4.7304$ &
\end{tabular}
\caption{Lower bounds $p^*_d$ and number of variables $\bar{n}$ of the LMI problem in function of relaxation order $d$ for the problem of \S\ref{sec:tankex}. The value found by hot starting \texttt{BOCOP} with our method is also shown.
}
\label{tab:tankex}
\end{table}

\subsection{Quadrotor}
\label{sec:quadrotorex}

For our last example, we take the quadrotor example considered in \cite{vasudevan2014consistent}:
\begin{equation}
\begin{aligned}
p^* & = & \inf  & \int_0^{T} \! 5 \, u^2(t) \, \diff t + \left\lVert \begin{smallmatrix} \sqrt{5} (x_1(T)-6) \\ \sqrt{5} (x_2(T)-1) \\ \sin(\frac{x_3(T)}{2}) \end{smallmatrix} \right\rVert_2^2 \\,
&& \mathrm{s.t.}\; & \ddot{x} = f_{\sigma(t)}, \\
&&& x(0) = (0,1,0), \dot{x}(0) = (0,0,0), \\
&&& u(t) \in U = [0,10^{-3}], \\
&&& \sigma(t) \in \lbrace 1, 2, 3 \rbrace, \\
&&& T=7.5,
\end{aligned}
\end{equation}
with data $M=1.3$, $L=0.3050$, $I=0.0605$, $g=9.8$, and
\begin{gather}
f_1 = 
\left(\begin{smallmatrix} \frac{\sin x_3(t)}{M} (u(t)+Mg) \\ \frac{\cos x_3(t)}{M} (u(t)+Mg)-g \\ 0 \end{smallmatrix} \right), \\
f_2 = \left(\begin{smallmatrix} g \sin x_3(t) \\ g \cos x_3(t)-g \\ \frac{-L u(t)}{I} \end{smallmatrix} \right), \quad
f_3 = \left(\begin{smallmatrix} g \sin x_3(t) \\ g \cos x_3(t)-g \\ \frac{L u(t)}{I} \end{smallmatrix} \right)
\end{gather}
{and where the infimum is w.r.t. a switching sequence $\sigma(t)$ but 
also a classical control $u(t)$, consistently with the extension sketched in \S \ref{sec:wider}.}
We turn this problem into one with polynomial data by considering another parametrization of the attitude $x_3$ belonging to $SO(2)$, namely  by replacing $x_3$ with $\ell_1:=\cos x_3$ and $\ell_2:=\sin x_3$, who must then satisfy the ODE
\begin{equation}
\begin{aligned}
\dot{\ell}_1 & = - \ell_2 \dot{x}_3,\\
\dot{\ell}_2 & =  \ell_1 \dot{x}_3.
\end{aligned}
\end{equation}
The problem has however too high of a dimension to hope to obtain  tight bounds systematically, with $7$ states/lifts, $1$ control and time (see the analysis in \S\ref{sec:complexity}). Nonetheless, for this specific example, the bounds presented in Tab.~\ref{tab:quadrotorex} show convergence in two relaxations only,
{since the local minimum obtained with \texttt{BOCOP} is very close to the
bound obtained at the relaxation of order $d=3$, solved in about 7 minutes on our desktop PC.}
Note that the relaxation of order {$d=4$, involving
$\bar{n}=98670$ moments,} would still be attainable with modern machines, at the expense of a full day of computations. 
Finally, also note that the solution is again better than the local minimum reported in \cite{vasudevan2014consistent}.

\begin{table}
\centering
\begin{tabular}{c|c|c}
$d$ & $p^*_d$ & $\bar{n}$  \\ \hline
2 & $9.0255 \, 10^{-3}$ & 3135 \\
3 & $9.4229 \, 10^{-2}$ & 21021 \\
\ldots & & \\
\texttt{BOCOP} & $9.5754 \, 10^{-2}$ &
\end{tabular}
\caption{Lower bounds $p^*_d$ and number of variables $\bar{n}$ of the LMI problem in function of relaxation order $d$ for the problem of \S\ref{sec:quadrotorex}. The value found by hot starting \texttt{BOCOP} with our method is also shown.
}
\label{tab:quadrotorex}
\end{table}

\section*{Acknowledgments:} This work benefited from discussions with Milan Korda,
Jean-Bernard Lasserre and Luca Zaccarian.

\end{document}